\documentclass[10pt,a4paper,oneside]{amsproc}

\usepackage{latexsym, amsthm}
\usepackage{amsfonts, amsmath, amssymb,comment}
\usepackage{euscript,mathrsfs}
\usepackage{url}
\usepackage{array}
\usepackage[a4paper]{geometry}
\usepackage{graphics}
\usepackage[usenames]{color}
\usepackage[all]{xy}
\usepackage{graphicx}
\usepackage{boxedminipage}

\newtheorem{theorem}{Theorem}[section]

\newtheorem{remark}[theorem]{Remark}

\newtheorem{lemma}[theorem]{Lemma}

\newtheorem{proposition}[theorem]{Proposition}

\newcommand{\B}{\mathcal{B}}

\newcommand{\Fq}{\mathbb{F}_q}

\newcommand{\LL}{\mathcal{L}}
\newcommand{\C}{\mathcal{C}}
\newcommand{\X}{\mathcal{X}}

\newcommand{\PRM}{\mathrm{PRM}_q}

\newcommand{\N}{\mathcal{N}}

\def\Fq{{\mathbb F}_q}
\def\AA{{\mathbb A}}

\def\PP{{\mathbb P}}

\begin{document}
\title{Remarks on second and third weights of Projective Reed-Muller codes}  
\author{Mrinmoy Datta}

\address{Department of Mathematics, \newline \indent
Indian Institute of Technology Hyderabad, Kandi, Sangareddy, Telangana, India}
\email{mrinmoy.datta@math.iith.ac.in}
\thanks{The first named author is partially supported by the grant DST/INT/RUS/RSF/P-41/2021 from the
Department of Science and Technology, Govt. of India} 


\date{}
\maketitle
{\centering\footnotesize Dedicated to Professor Sudhir Ghorpade on his 60th birthday.\par}
\begin{abstract}   
Determining the weight distributions of the projective Reed-Muller codes is a very hard problem and has been studied extensively in the literature. In this article, we provide an alternative proof of the second weight of the projective Reed-Muller codes $\PRM (d, m)$ where $m \ge 3$ and $3 \le d \le \frac{q+3}{2}$. We show that the second weight is attained by codewords that correspond to hypersurfaces containing a hyperplane under the hypothesis on $d$. Furthermore, we compute the second weight of $\PRM (d, 2)$ for $3 \le d \le q$.  Furthermore, we give an upper bound for the third weight of $\PRM(d, 2)$.
\end{abstract}
\section{Introduction}
The generalized Reed-Muller codes and the projective Reed-Muller codes have been studied widely by coding theorists and mathematicians in the last few decades. While the Reed-Muller codes were introduced by D.E. Muller and I.S. Reed in 1954, the so-called generalized Reed-Muller codes, which are one of the central objects of study in this article, were introduced by Kasami, Lin, and Peterson \cite{KLP} in 1968. While the above-mentioned codes have been used for practical purposes in coding theory, they have also given rise to fundamental problems in algebraic geometry, combinatorics, number theory, and many other fields. A projective analog of the generalized Reed-Muller codes, namely, the projective Reed-Muller codes, was introduced by Lachaud \cite{L} in 1988. An introduction to the codes mentioned above is given in Section \ref{introtocode}. 

As explained later, the determination of weight distributions of the generalized Reed-Muller codes $RM(d, m)$ and the projective Reed-Muller codes $PRM(d, m)$ are fascinating problems in mathematics. In this paper, we restrict our attention to the case when $1 \le d < q$. In particular, the Hamming weights of the Reed-Muller codes correspond to the number of $\Fq$-rational points on hypersurfaces defined over $\Fq$. The minimum weight of the Reed-Muller codes is well-known due to a very old result by Ore (cf Theorem \ref{one}), while the second weight of the Reed-Muller codes is known due to Olav Geil \cite{G}. The minimum weight of projective Reed-Muller codes was determined by Serre \cite{S} and independently by S{\o}rensen \cite{So}.  Much work has been done to compute the second weight of the projective Reed-Muller codes. They can be found in \cite{BR, CN, CN1, CR, RS, Sbspec, Sb2nd} among others. The results in the above-mentioned literature will be discussed later in this paper. However, the complete determination of the second weight of projective Reed-Muller codes is beyond reach at this moment. The complete picture of the progress made towards the problem of determination of the second weight of projective Reed-Muller codes could be obtained from Table 3 in \cite{CN1}. Similarly, some attempts have been made towards finding the third weight of projective Reed-Muller codes in \cite{RS, Sbspec}. 

In this article, we completely determine the second weight of $PRM (d, m)$ when $m = 2$. As it will be explained later, this amounts to the determination of the second highest number of $\Fq$-rational points on a projective plane curve of degree $d$. Moreover, we also compute an upper bound for the third highest possible number of $\Fq$-rational points on a projective plane curve of degree $d$. We remark that the second weight of $PRM (d, m)$ is known for $d \le \frac{q+3}{2}$, thanks to Carvalho and Neumann \cite{CN, CN1}. We give an alternative proof of this result which also shows that for $d \le \frac{q+3}{2}$, the codewords that attain the second weight correspond to hypersurfaces that are union of hyperplanes. Along the way, we also provide a mild improvement of the Homma-Kim elementary bound for hypersurfaces of degree $d \neq 2, \sqrt{q} + 1$. 

This article is organized as follows: In Section 2, we recall various bounds on the number of $\Fq$-rational points on varieties defined over $\Fq$, with a strong emphasis on projective plane curves of given degree. We also recall the definitions of the Reed-Muller codes in a small subsection to make the connections with our work more clear. In Section 3, we provide a minor improvement of the so-called Homma-Kim elementary bound. 
Finally, in Section 4, we determine the second highest number of $\Fq$-rational points on projective plane curves of degree $d$, determining the second weight of the corresponding projective Reed-Muller codes. Furthermore, we provide an upper bound on the third highest number of $\Fq$-rational points on projective plane curves.     

\section{Notations and Preliminary results}
Let $\Fq$ denote a finite field with $q$ elements, where $q = p^e$, for some prime number $p$ and a positive integer $e$. For $m \ge 1$, we will denote by $\PP^m$ (resp. $\AA^m)$ the $m$-dimensional projective space (resp. affine space) defined over $\overline{\Fq}$, an algebraic closure of $\Fq$. The set of $\Fq$-rational points in $\PP^m$ (resp. $\AA^m$) will be denoted by $\PP^m (\Fq)$ (resp. $\AA^m (\Fq)$). By a projective hypersurface of degree $d$ defined over $\Fq$ in $\PP^m$, we will mean the zero set of a homogeneous polynomial over $\Fq$ in $m+1$  variables of degree $d$ in $\PP^m$. If $\X$ is a projective hypersurface in $\PP^m$, we will denote by $\X(\Fq)$ the set of all $\Fq$-rational points on $\X$. We will use similar notations to describe affine hypersurfaces. Needless to say, a plane projective (resp. affine) curve is a hypersurface in $\PP^2$ (resp $\AA^2)$. We often denote by $\widehat{\PP^m (\Fq)}$ the set of all projective hyperplanes in $\PP^m$ defined over $\Fq$. For $m \ge 0$, we will use the notation $p_m$ to denote the cardinality of $\PP^m(\Fq)$, that is, $p_m = q^m + \dots + 1$.

Throughout this article, all our varieties (curves, hypersurfaces, etc.) and polynomials are defined over $\Fq$ unless otherwise mentioned.  We are interested in the determination of good upper bounds on the number of $\Fq$-rational points on projective (resp. affine) hypersurfaces of degree $d$ defined over $\Fq$ (as mentioned earlier, we will omit to mention the field of definition of varieties from now on) in $\PP^m$ (resp. $\AA^m$). In what follows, we will always assume that $2 \le d \le q$.  We devote the rest of the section to revisiting the known results related to the number of $\Fq$-rational points on varieties over finite fields.

\subsection{Rational points on affine hypersurfaces over finite fields}
It is well-known that a nonzero polynomial of degree $d$ in one variable with coefficients in a field $k$ has at most $d$ roots in $k$. When it comes to finite fields, that is, when $k = \Fq$, it is evident that a polynomial in $m$ variables with coefficients in $k$ has finitely many $\Fq$-rational zeroes in an $m$-dimensional affine space over $\Fq$. The following result goes back to {\O}. {O}re, and is found in \cite[Theorem 6.13]{LN}.

\begin{theorem}\cite[Theorem 6.13]{LN}\label{ore}
    If $d < q$, then an affine hypersurface in $\AA^m$ of degree $d$ 
    has at most $dq^{m-1}$ many $\Fq$-rational points. 
\end{theorem}

The bound in Theorem \ref{ore} is attained and corresponds to the minimum distance of the generalized Reed-Muller codes. The next result, towards determination of the second weight of generalized Reed-Muller codes, was proved by Geil in \cite[Proposition 2]{G}. 
 
\begin{theorem}\cite[Proposition 2]{G}\label{G}
    Suppose $m \ge 2$ and $S$ be an affine hypersurface in $\AA^m$ of degree at most 
    $d$. 
    If $2 \le d < q$ and $|S(\Fq)| < dq^{m-1}$, then $|S(\Fq)| \le dq^{m-1} - (d-1)q^{m-2}$. 
\end{theorem}

While Theorem \ref{G} appeared in 2008, the upper bound for large values of $q$ was established in 1995 by Cherdieu and Rolland \cite{CR}. They also proved the following result.  

\begin{theorem}\cite[Theorem 2.1]{CR}\label{CR}
    Let $m \ge 2$ and $S$ be an affine hypersurface of degree $d$ in $\AA^m$ 
    that is a union of $d$ affine hyperplanes. If $|S(\Fq)| = dq^{m-1} - (d-1)q^{m-2}$, then
    \begin{enumerate}
        \item[(a)] either \textbf{Type I:} $S$ is a union of $d$ affine hyperplanes 
        passing through a common codimension $2$ affine subspace 
        in $\AA^m$, or
        \item[(b)] \textbf{Type II:} $S$ is a union of $d-1$ parallel hyperplanes 
        and another hyperplane 
        intersecting each of the $d-1$ hyperplanes at distinct codimension $2$ affine subspaces. 
    \end{enumerate}
\end{theorem}
Finally, it was proved by Leducq \cite{Le} that the second weight codewords of generalized Reed-Muller codes are given by a hypersurface that is a union of hyperplanes. We record the result by Leducq in our language for easy reference.

\begin{theorem}\cite[Proposition 6]{Le}\label{leducq}
    Let $m \ge 2$  and \ $2 \le d \le q-1$. If $S$ is an affine hypersurface of degree $d$ defined over $\Fq$ such that $|S(\Fq)| = dq^{m-1} - (d-1)q^{m-2}$, then $S$ is either of Type I or of Type II. 
\end{theorem}

\subsection{Rational points on projective hypersurfaces over finite fields}
Having revisited some of the important results concerning affine hypersurfaces over finite fields, we now turn our attention to projective hypersurfaces. The first result to this end, resolving a conjecture by Tsfasman, was proved by Serre in \cite{S}. The result was independently proved by S{\o}rensen in \cite{So}. We refer the reader to \cite{DG} for a more recent proof of the same. 

\begin{theorem}[Serre's inequality, \cite{S}]\label{S}
    Let $V$ be a projective hypersurface of degree $d$ in $\PP^m$. 
    If  $1 \le d \le q$, then $|V(\Fq)| \le dq^{m-1} + p_{m-2}$.
\end{theorem}

\noindent It is implicit in the proof of Serre's inequality that the bound is attained only by hypersurfaces of degree $d$ that are a union of $d$ distinct hyperplanes in $\PP^m$ containing a common projective linear subspace of codimension two. In particular, when $m = 2$, Serre's inequality says that a projective plane curve of degree $d$, with $d \le q$, has at most $dq + 1$ many $\Fq$-rational points. Furthermore, the bound is attained by the curves of degree $d$ that are a union of $d$ distinct lines intersecting at a common point. 

In a series of three papers \cite{HK1, HK2, HK3}, Homma and Kim proved the following Theorem, settling completely a conjecture made by P. Sziklai towards the maximum number of points on a projective plane curve that does not contain a line. This result will play a crucial role in deriving our results in the latter part of this article. 

\begin{theorem}[\cite{HK1, HK2, HK3}]\label{hk}
    Let $\C$ be a projective plane curve of degree $d$ defined over $\Fq$ containing no lines defined over $\Fq$. If $2 \le d \le q$ and $(d, q) \neq (4, 4)$, then $|\C (\Fq)| \le (d-1)q + 1$.
\end{theorem}

\begin{remark}\label{attained}\normalfont
    For $2 \le d \le q-1$, it is evident from the proofs in \cite{HK2, HK3}, that the curves of degree $d$ attaining the bound in Theorem \ref{hk} can be divided into two main classes. 
    \begin{enumerate}
        \item[(a)]  \textit{Frobenius Classical:} In this case, the degrees of such curves are either $2$ or $q-1$. In the first case, when the curve is of degree two, it is, in fact, a nonsingular conic. 
        \item[(b)] \textit{Frobenius Non-classical:} In this case, if $q = p^e$, as mentioned above, the possible degrees of curves attaining the bounds are $p^{h} + 1$ for some $\frac{e}{2} \le h < e$. 
    \end{enumerate} 
    It is beyond the scope of this article to go into the theory of Frobenius classical curves. We refer the readers to \cite{HV, SV, HK2, HK3} and the references therein for a detailed understanding of these notions. Let us illustrate the above-mentioned criteria in two particular cases. First, let us assume that $p$ is an odd prime. Also, suppose that $\C$ is a projective plane curve of degree $d$, with $2 \le d <q-1$ and $|\C(\Fq)| = (d-1)q + 1$. 
    \begin{enumerate}
        \item[(i)] If $q = p$, it follows that $\C$ must be a nonsigular conic.
   \item[(ii)] If $q = p^2$, then either $d = 2$ or $d = p+1$. It follows as above, that if $d = 2$, then $\C$ is a nonsingular conic. Hirschfeld, Storme, Thas and Voloch showed that \cite[Theorem 8]{HSTV}  if $q$ is a perfect square with $q \neq 4$, then any plane curve of degree $\sqrt{q}+1$ that does not contain a line and admits $q^{\frac{3}{2}} + 1$ many $\Fq$-rational points is a nonsingular Hermitian curve. Consequently, when $d = p + 1$, the curve $\C$ is a nonsingular Hermitian curve. 
    \end{enumerate} 
    We remark that classifying all the projective plane curves attaining the bound in Theorem \ref{hk} is a very interesting open problem. 
\end{remark}

When $m \ge 3$, the following result, popularly known as the \emph{Homma-Kim elementary bound} \cite[Theorem 1.2]{HKel}, gives an upper bound on the number of $\Fq$-rational points that a projective hypersurface of degree $d$ may have in $\PP^m$. This bound also occurs, albeit in a different avatar, in \cite[Proposition 1]{CN}.

\begin{theorem}\cite[Theorem 1.2]{HKel}\label{hkel}
    Let $\X$ be a projective hypersurface in $\PP^m$ of degree $d$. 
    If $\X$ does not contain a hyperplane defined over $\Fq$, then $|\X(\Fq)| \le (d-1)q^{m-1} + dq^{m-2} + p_{m-3}$. 
\end{theorem}
\noindent In the same paper, the authors illustrated some examples of hypersurfaces that attain the elementary bound. Later, Tironi \cite{T} determined all the hypersurfaces without a linear component reaching the bound. Since we are mostly interested in hypersurfaces of degree at most $q-1$, we partially record his result below. 

\begin{theorem}\cite[Theorem 1]{T}\label{TT}
    Let $2 \le d \le q-1$, $m \ge 3$, and $\X$ be a projective hypersurface of degree $d$ in $\PP^m$ that attains the Homma-Kim elementary bound. 
    Then one of the following holds: 
    \begin{enumerate}
        \item[(a)] $d=2$, and $\X$ is projectively equivalent to a cone over the hyperbolic quadric surface, i.e., a surface in $\PP^3$ given by the vanishing set of $x_0x_2 + x_1x_3$. 
        \item[(b)] $q$ is a square and $d = \sqrt{q} + 1$, and $\X$ is projectively equivalent to a cone over the nondegenerate Hermitian surface, i.e., a surface in $\PP^3$ given by the vanishing set of $x_0^{\sqrt{q} + 1}+ x_1^{\sqrt{q}+1} + x_2^{\sqrt{q} + 1}+ x_3^{\sqrt{q}+1}$. 
    \end{enumerate}
\end{theorem}
Before we conclude this subsection, we revisit a very useful Lemma by Zanella. 
\begin{theorem}[Zanella's Lemma] \cite{Z}
Let $S \subset \PP^m (\Fq)$ and $a = \max \{|S(\Fq) \cap \Pi| : \Pi \in \widehat{\PP^m(\Fq)}\}$, then $|S(\Fq)| \le aq + 1$.
\end{theorem}

\subsection{(Projective) Reed-Muller Codes}\label{introtocode}
An $[n, k]_q$  linear code (or simply a code, when $n, k, q$ are clear from the context) $C$ is a $k$-dimensional linear subspace of $\Fq^n$. On the vector space $\Fq^n$, one defines a metric, known as the Hamming metric, as follows: 
For ${\bf a} = (a_1, \dots, a_n), {\bf b} = (b_1, \dots, b_n) \in \Fq^n$, we define the Hamming distance between ${\bf a}$ and ${\bf b}$, denoted $d({\bf a}, {\bf b})$ as the number of coordinates where they differ, namely, $d ({\bf a}, {\bf b}) = \#\{i : a_i \neq b_i\}$. If $C$ is a code, then the minimum distance of $C$, denoted by $d (C)$, is defined as 
$$d(C) :=\min \{ d ({\bf a}, {\bf b}) : {\bf a}, {\bf b} \in C, {\bf a} \neq {\bf b} \}.$$
Given an element, also known as a codeword, ${\bf a} \in C$, we define the Hamming weight of ${\bf a}$, denoted $w ({\bf a})$, as $d ({\bf a}, 0)$, that is, the number of nonzero coordinates of ${\bf a}$. The minimum weight of a linear code $C$, as clear from the nomenclature, is defined as the minimum of the Hamming weights of all nonzero codewords of $C$. Using the linearity of the codes, it is easy to prove that the minimum weight and the minimum distance of a code are equal. An $[n, k, d]_q$ linear code with minimum weight $d$ is called an $[n, k, d]_q$ linear code. 

\subsubsection{Generalized Reed-Muller codes}
Let $m, d$ be positive integers and $n = q^m$. In this article, we will always assume that $1 \le d \le q-1$. We fix an enumeration $\{P_1, \dots, P_n\}$ of the elements of $\AA^m (\Fq)$. Let $R_{\le d}$ be the set of all polynomials in  $\Fq[x_1, \dots, x_m]$ of total degree at most $d$. This makes $R_{\le d}$ a vector space over $\Fq$. Define an evaluation map
$$ev : R_{\le d} \to \Fq^n \ \ \ \text{as} \ \ \ f \mapsto ev (f) := (f(P_1), \dots, f(P_n)).$$
The map $ev$ is a linear map and the image of $ev$ is a linear subspace of $\Fq^n$. We define the generalized Reed-Muller code, denoted by $RM (d, m)$ as $ev (R_{\le d})$. Note that the determination of the Hamming weight of a codeword $ev (f)$ of $RM(d, m)$ is equivalent to the determination of the number of $\Fq$-rational zeroes of $f$. Consequently, the minimum weight of $RM (d, m)$ is given by $q^m - dq^{m-1}$, thanks to Theorem \ref{ore}. Similarly, the second minimum weight of $RM (d, m)$, that is, the smallest weight of a codeword having Hamming weight larger than the minimum weight, is given by $q^m - dq^{m-1} + (d-1)q^{m-2}$, as a consequence of Theorem \ref{G}. Determining the weight distribution of the Reed-Muller codes is undoubtedly a fascinating problem. Mathematically speaking, this reduces the problem of determining all possible numbers of $\Fq$-rational zeroes of a polynomial of degree at most $d$ in $m$ variables. 

\subsubsection{Projective Reed-Muller codes}
Let $m, d$ be positive integers and $n = p_m:= q^m + q^{m-1} + \dots + 1$. In this article, we will always assume that $1 \le d \le q-1$. We fix an enumeration $\{P_1, \dots, P_n\}$ of the elements of $\PP^m (\Fq)$. Furthermore, we fix a representative of each point in $\PP^m (\Fq)$, by saying that the first nonzero coordinate of each point is $1$.
Let $S_{d}$ be the set of all homogeneous polynomials in  $\Fq[x_0, x_1, \dots, x_m]$ of total degree  $d$. This makes $S_{d}$ a vector space over $\Fq$. Define an evaluation map
$$Ev : S_d \to \Fq^n \ \ \ \text{as} \ \ \ F \mapsto Ev (F) := (F(P_1), \dots, F(P_n)).$$
The map $ev$ is a linear map, and the image of $ev$ is a linear subspace of $\Fq^n$. We define the projective Reed-Muller code, denoted by $PRM (d, m)$ as $Ev (S_d)$. As in the case with the generalized Reed-Muller codes, the Hamming weight of a codeword $Ev (F)$ of $PRM(d, m)$ is intrinsically related to the number of zeroes of $F$ in $\PP^m (\Fq)$. Consequently, the minimum weight of $PRM (d, m)$ is given by $p_m - dq^{m-1} - p_{m-2}$, thanks to Theorem \ref{S}. However, the complete determination of the second minimum weight of $PRM (d, m)$, the smallest weight of a codeword of $PRM (d, m)$ having a Hamming weight larger than the minimum weight, is not completely solved yet. One of the prominent goals of this article is to determine the second weight of $PRM (d, m)$ for $m = 2$. 

\section{Hypersurfaces}
In this section, we consider the set $\mathscr{S}$ of hypersurfaces of degree $d$, with $2 \le d \le q-1$ in $\PP^m$ that does not attain the bound in Serre's inequality and produce an upper bound for the number of points on hypersurfaces that belong to $\mathscr{S}$. Sboui showed \cite[Theorem 4.2]{Sbspec} that if $5 \le d \le \frac{q+2}{2}$, then 
$$|\X(\Fq)| \le dq^{m-1} + p_{m-2} - (d-2) q^{m-2},$$
for all $\X \in \mathscr{S}$. Furthermore, in the same paper, he proved that this bound is attained by a union of $d$ distinct hyperplanes. Carvalho and Neumann further investigated this in \cite{CN1}, where they showed that the upper bound holds for $2 \le d \le \frac{q+3}{2}$. It is not known yet whether, in the case $d = \frac{q+3}{2}$, the bound is still attained by hypersurfaces that are the union of $d$ distinct hyperplanes. In this article, we give an alternative proof of the above bound for $2 \le d \le \frac{q+3}{2}$ using Theorem \ref{G} and a rudimentary improvement of Theorem \ref{hkel}. Moreover, our proof also shows that the bound is attained by hypersurfaces that contain at least one hyperplane. We begin with the following simple observation, the idea of which is well-known among experts and appears in several forms in literature, for example, in \cite[Lemma 2.3(i)]{CN1}. We include a proof for the sake of completeness. 
\begin{proposition}\label{hyp}
Suppose $m \ge 2$ is an integer. Let $\X$ be a hypersurface in $\PP^m$ of degree $d$, with $3 \le d \le q$. 
 If $\X$ contains a hyperplane 
 and $|\X(\Fq)| < dq^{m-1} + p_{m-2}$, then $$|\X(\Fq)| \le dq^{m-1} + p_{m-2} - (d-2)q^{m-2}.$$  
 Moreover, if $\X$ is a hypersurface of degree $d$ such that $\X$ contains a hyperplane and $|\X(\Fq)| = dq^{m-1} + p_{m-2} - (d-2)q^{m-2}$, then $\X$ is a union of hyperplanes. 
\end{proposition}

\begin{proof}
  Let $\Pi$ be a hyperplane contained in $\X$. 
With a suitable linear change of coordinates, we may assume that $\Pi$ is the hyperplane at infinity. Then $\X' := \X \setminus \Pi$ is an affine hypersurface of degree at most $d-1$ in $\AA^m$. 
  Note that, 
  \begin{equation}\label{geil}
  |\X(\Fq)| = |\Pi (\Fq)| + |\X'(\Fq)|.
  \end{equation}
  Theorem \ref{ore} implies $|\X' (\Fq)| \le (d-1)q^{m-1}$. However, if the equality occurs, then $|\X(\Fq)| = dq^{m-1} + p_{m-2}$, which contradicts our hypothesis. Thus $|\X'(\Fq)| < (d-1)q^{m-1}$. Since $2 \le d - 1 \le q-1$, Theorem \ref{G} implies $|\X'(\Fq)| \le (d-1)q^{m-1} - (d-2)q^{m-2}$. Using \eqref{geil}, we see that $|\X (\Fq)| \le dq^{m-1} + p_{m-2} - (d-2)q^{m-2}$, as desired. Note that equality occurs if and only if $|\X'(\Fq)| = (d-1)q^{m-1} - (d-2)q^{m-2}$. Now the last assertion follows from Theorem \ref{leducq}.
\end{proof} 

\begin{remark}\label{HKsuffice}\normalfont
Note that the Homma-Kim elementary bound is stronger than the one in the above proposition when $3 \le d \le \frac{q+3}{2}$. Indeed, we have
\begin{align*}
 &\left(dq^{m-1} + p_{m-2} - (d-2)q^{m-2}\right) - \left((d-1)q^{m-1} + dq^{m-2} + p_{m-3}\right) \\
&=q^{m-1} + q^{m-2} - 2(d-1)q^{m-2} \\
&= q^{m-2} (q + 3 - 2d) \ge 0
\end{align*}
when $d \le \frac{q+3}{2}$. Moreover, if $d \le \frac{q+2}{2}$, then
\begin{equation}\label{strict}
   (d-1)q^{m-1} + dq^{m-2} + p_{m-3} < dq^{m-1} + p_{m-2} - (d-2)q^{m-2}.
\end{equation}
As explained in Theorem \ref{TT}, the Homma-Kim elementary bound is attained only in the cases when $d = 2$ and $d= \sqrt{q} + 1$, when $q$ is a perfect square. Thus, under the hypothesis $3 \le d \le q$, the Homma-Kim elementary bound can be attained only when $d = \sqrt{q} + 1$. Since $\sqrt{q} + 1 < \frac{q+3}{2}$, the strict inequality in \eqref{strict} is satisfied. As a consequence, the Homma-Kim elementary bound and Proposition \ref{hyp} shows that
if $3 \le d \le \frac{q+3}{2}$, and $\X$ is a hypersurface of degree $d$ in $\PP^m$ such that $|\X(\Fq)| < dq^{m-1} + p_{m-2}$, then $|\X (\Fq)| \le dq^{m-1} + p_{m-2} - (d-2)q^{m-2}$. Moreover, if $d \le \frac{q+2}{2}$, then bound is attained by $\X$ only if $\X$ contains a hyperplane. We note that this was already proved by Sboui (cf. \cite[Theorem 4.2]{Sbspec}) using a different methodology. We will extend these observations to the case $d \le \frac{q+3}{2}$ using an improvement of Homma-Kim elementary bound.
\end{remark}

We now present an improvement of the Homma-Kim elementary bound in the cases when $d \neq 2, \sqrt{q} + 1$. To warm up, let us first investigate the case when $m = 3$. 

\begin{lemma}\label{propfor3}
  Let $3 \le d \le q$ and  $\X$ be a surface of degree $d$ in $\PP^3$ containing no planes. If  $d \neq \sqrt{q} + 1$, then $$|\X(\Fq)| \le (d-1)q^{2} + dq + 1 - (d-2).$$      
\end{lemma}

\begin{proof}
    Let $\X$ be a surface in $\PP^3$ and $\deg \X = d$. If $\X$ contains a line $\ell$ defined over $\Fq$, then we see that 
    \begin{equation}\label{one}
        |\X(\Fq)| = |\ell(\Fq)| + \sum_{\Pi \in \B(\ell)} |\X(\Fq) \cap (\Pi \setminus \ell)|, 
    \end{equation}
    where $\B(\ell)$ denotes the set of all planes in $\PP^3$ containing $\ell$. Note that $|\B(\ell)| = q+1$. Let $\Pi \in \B(\ell)$. Then $\Pi \setminus \ell$ is an affine plane. Since $d \le q$, it follows from our hypothesis that $\X$ does not contain $\Pi \setminus \ell$. Furthermore, since $\X$ contains $\ell$, we see that $\X \cap (\Pi \setminus \ell)$ is an affine curve of degree at most $d-1$. Now Theorem \ref{ore} implies that $|\X(\Fq) \cap (\Pi \setminus \ell)| \le (d-1)q$. If it so happens that $|\X(\Fq) \cap (\Pi \setminus \ell)| = (d-1)q$ for all $\Pi \in \B(\ell)$, then we see from \eqref{one} that $|\X(\Fq)|=(d-1)q^2 + dq + 1$, contradicting our hypothesis in view of  Theorem \ref{TT}. Thus there exists  $\Pi_0 \in \B(\ell)$ such that $|\X(\Fq) \cap (\Pi_0 \setminus \ell)| < (d-1)q$. From Theorem \ref{G}, we have $|\X(\Fq) \cap (\Pi_0 \setminus \ell)| \le (d-1)q  - (d-2)$. Consequently, from \eqref{one}, we have
    $$|\X(\Fq)| \le (q+1) + (d-1)q - (d-2) + q(d-1)q = (d-1)q^2 + dq + 1 - (d-2).$$
    Now, suppose that $\X$ does not contain any line. So for each plane $\Pi$ in $\PP^3$, we see that $\X \cap \Pi$ is a plane curve of degree $d$ not containing any lines. Since $d \le q$, Theorem \ref{hk} shows that $|\X(\Fq) \cap \Pi| \le (d-1)q+1$ for each plane $\Pi$ in $\PP^3$ defined over $\Fq$. Using Zanella's Lemma, we see that $$|\X(\Fq)| \le ((d-1)q+1)q+1 = (d-1)q^2 + q + 1 < (d-1)q^2 + dq + 1 - (d-2),$$ as desired. This completes the proof. 
\end{proof}
We are now ready to state and prove the main result of this section. 
\begin{proposition}\label{rudim}
    Suppose $3 \le d \le q$ and $m \ge 3$. Let $\X$ be a hypersurface of degree $d$ in $\PP^m$  containing no hyperplanes. If  $d \neq \sqrt{q} + 1$, then $$|\X(\Fq)| \le (d-1)q^{m-1} + dq^{m-2} + p_{m-3} - (d-2)q^{m-3}.$$      
\end{proposition}

\begin{proof}
    We prove this by induction on $m$. When $m=3$, the assertion follows from Proposition \ref{propfor3}. Thus we may assume that $m >3$ and our assertion is true for all projective spaces of dimension less than $m$. Let $\X$ be a hypersurface of degree $d$ defined over $\Fq$ in $\PP^m$ without any $\Fq$-linear component. We split the proof into two cases:

    \textbf{Case 1: $\X$ contains a linear subspace of codimension $2$.} If $\X$ contains a linear subspace $\LL$ of codimension two, then we see that 
    \begin{equation}\label{two}
        |\X(\Fq)| = |\LL(\Fq)| + \sum_{\Pi \in \B(\LL)} |\X(\Fq) \cap (\Pi \setminus \LL)|, 
    \end{equation}
    where $\B(\LL)$ denotes the set of all hyperplanes in $\PP^m$ containing $\LL$ defined over $\Fq$. Thus $|\B(\LL)| = q+1$. Suppose that $\Pi \in \B(\LL)$. Then $\Pi \setminus \LL$ is an affine space of dimension $m-1$, and since $d \le q$ by hypothesis, $\X$ does not contain $\Pi \setminus \LL$. Furthermore, since $\X$ contains $\LL$, it follows that $\X \cap (\Pi \setminus \LL)$ is an affine hypersurface of degree at most $d-1$. Now, Theorem \ref{ore} implies that $|\X(\Fq) \cap (\Pi \setminus \LL)| \le (d-1)q^{m-2}$. If $|\X(\Fq) \cap (\Pi \setminus \LL)| = (d-1)q^{m-2}$ for all $\Pi \in \B(\LL)$, then we see from \eqref{two} that $|\X(\Fq)|=(d-1)q^{m-1} + dq^{m-2} + p_{m-3}$. Now Theorem \ref{TT} implies that $d$ must equal $2$ or $\sqrt{q} + 1$, contradicting our hypothesis. Thus there exists  $\Pi_0 \in \B(\ell)$ such that $|\X(\Fq) \cap (\Pi_0 \setminus \ell)| < (d-1)q^{m-2}$. Since $2 \le d-1 \le q-1$, Theorem \ref{G} applies and shows that $|\X(\Fq) \cap (\Pi_0 \setminus \ell)| \le (d-1)q^{m-2}  - (d-2)q^{m-3}$. Consequently, from \eqref{two}, we have
    $$|\X(\Fq)| \le p_{m-2} + (d-1)q^{m-2} - (d-2)q^{m-3} + q(d-1)q^{m-2}= (d-1)q^{m-1} + dq^{m-2} + 1 - (d-2)q^{m-3}.$$

    \textbf{Case 2: $\X$ does not contain a linear subspace of codimension $2$.} In this case, for each hyperplane $\Pi$ in $\PP^m$ defined over $\Fq$, we see that $\X \cap \Pi$ is a hypersurface of degree $d$ not containing any hyperplanes in $\Pi \cong \PP^{m-1}$ defined over $\Fq$. From induction hypothesis, we have $|\X(\Fq) \cap \Pi| \le (d-1)q^{m-2} + dq^{m-3} + p_{m-4} - (d-2)q^{m-4}$ for each hyperplane $\Pi$ in $\PP^m$ defined over $\Fq$. From Zanella's lemma, we see that 
    \begin{align*}
    |\X(\Fq)| &\le ((d-1)q^{m-2} + dq^{m-3} + p_{m-4} - (d-2)q^{m-4})q+1 \\
    &= (d-1)q^{m-1} + dq^{m-2} + p_{m-3} - (d-2)q^{m-3}.
    \end{align*} 
    This completes the proof. 
    \end{proof}

\begin{remark}\label{rem:hypersurface}\normalfont
\
\begin{enumerate}
    \item[(a)] First, as observed before, $(d-1)q^{m-1} + dq^{m-2} + p_{m-3} \le dq^{m-1} + p_{m-2} - (d-2)q^{m-2}$ whenever $3 \le d \le \frac{q+3}{2}$. Furthermore, equality holds only when $d = \frac{q+3}{2}$. 
    
    \item[(b)] Note that $(d-1)q^{m-1} + dq^{m-2} + p_{m-3} - (d-2)q^{m-3} < dq^{m-1} + p_{m-2} - (d-2)q^{m-2}$, when $3 \le d \le \frac{q+3}{2}$. Indeed, for $3 \le d \le \frac{q+3}{2}$, we have
\begin{align*}
    &\left(dq^{m-1} + p_{m-2} - (d-2)q^{m-2} \right) - \left((d-1)q^{m-1} + dq^{m-2} + p_{m-3} - (d-2)q^{m-3} \right)\\
    & = q^{m-1} - (2d - 3)q^{m-2}+ (d-2) q^{m-3} > q^{m-2} (q - 2d + 3) \ge 0.
\end{align*}
The strict inequality above results from the assumption that $d > 2$.  
\end{enumerate}
\end{remark}

We conclude the section by summarizing all the results presented in this section. 
\begin{theorem}\label{thm:3main}
    Suppose $q \ge 3$, $d \ge 3$, and $m \ge 3$. Let $\X$ be a hypersurface of degree $d$ defined over $\Fq$ in $\PP^m$. Moreover, assume that $|\X(\Fq)| < dq^{m-1} + p_{m-2}$. Then we have the following: 
    \begin{enumerate}
        \item[(a)] $|\X(\Fq)| \le dq^{m-1} + p_{m-2} - (d-2)q^{m-2}$, when $3 \le d \le \frac{q+3}{2}$. Moreover, equality holds only if $\X$ is a union of hyperplanes.
        \item[(b)] $|\X(\Fq)| \le (d-1)q^{m-1} + dq^{m-2} + p_{m-3} - (d-2)q^{m-3}$, whenever $\frac{q+3}{2} < d \le q$.  
    \end{enumerate}
    \end{theorem}

    \begin{proof}
        Let $\X$ be a hypersurface satisfying the hypothesis. If $\X$ contains a hyperplane defined over $\Fq$, then Proposition \ref{hyp} shows that $|\X(\Fq)| \le dq^{m-1} + p_{m-2} - (d-2)q^{m-2}$. If $\X$ does not contain any hyperplane defined over $\Fq$, then Proposition \ref{rudim} implies that $|\X(\Fq)| \le (d-1)q^{m-1} + dq^{m-2} + p_{m-3} - (d-2)q^{m-3}.$ Thus 
        \begin{equation}\label{eqmax2}
            |\X(\Fq)| \le \max \{dq^{m-1} + p_{m-2} - (d-2)q^{m-2}, (d-1)q^{m-1} + dq^{m-2} + p_{m-3} - (d-2)q^{m-3} \}.
        \end{equation}
         The inequalities in (a) and (b) now follow from Remark \ref{rem:hypersurface}. Suppose $\X \subset \PP^m$ be a hypersurface of degree $d$, with $3 \le d \le \frac{q+3}{2}$, defined over $\Fq$ such that 
         $$|\X(\Fq)| = dq^{m-1} + p_{m-2} - (d-2)q^{m-2}.$$
         Remark \ref{rem:hypersurface} shows that $\X$ must contain a hyperplane. 
         The assertion now follows from the last part of Proposition \ref{hyp}.
    \end{proof}

    It is conceivable that a major improvement of Homma-Kim elementary bound in all cases when $d \neq 2$ or $\sqrt{q}+1$, combined with Theorem \ref{G} will answer the question for determining all the hypersurfaces in $\PP^m$, with $m \ge 3$, that admit the second highest number of $\Fq$-rational points. 
\section{Plane curves}
In the previous section, we have partially answered the question of determination of the second highest number of $\Fq$-rational points on hypersurfaces of a given degree. It turns out that we can answer the question completely for projective plane curves. Furthermore, we present some new results on the third highest number of $\Fq$-rational points on projective plane curves. Throughout, we assume that $3 \le d \le q$. 

\subsection{Second highest number of $\Fq$-rational points on projective plane curves}
Let us denote by $\mathscr{C}$ the set of projective plane curves  of degree $d$.
It follows from Theorem \ref{S} that a curve in $\mathscr{C}$ admits at most $dq + 1$ many $\Fq$-rational points in $\PP^2$. As mentioned before, any curve attaining this bound will be given by a union of $d$ distinct lines having a point in common. Let us denote by $\mathscr{C}_1$ the set of all such curves. 
As evident from the context, we want to determine the highest number of points a curve $\C \in \mathscr{C} \setminus \mathscr{C}_1$ may admit. Sboui proved \cite[Theorem 4.2]{Sbspec} that $|\C(\Fq)| \le dq - d + 3$, for any $\C \in \mathscr{C} \setminus \mathscr{C}_1$ with $5 \le d \le \frac{q+2}{2}$.  We show that this bound works for all curves of degree at most $q - 1$.   

\begin{theorem}\label{thm2}
    Suppose $\C$ is a plane projective curve of degree $d$, with $3 \le d \le q$. If $|\C(\Fq)| < dq + 1$, then $|\C(\Fq)| \le dq - d + 3$. 
\end{theorem}

\begin{proof}
    If $\C$ contains a line defined over $\Fq$, then Proposition \ref{hyp} applies with $m = 2$ and proves the assertion. If $\C$ does not contain any line defined over $\Fq$, then Theorem \ref{hk} shows that $|\C(\Fq)| \le (d-1)q+ 1$. Note that
    $$(dq - d + 3) - \left((d-1)q + 1 \right) = q - d + 2 > 0.$$
    The last equality is true since $d \le q-1$ by hypothesis. 
\end{proof}

We now characterize all curves of degree $d$ with $3 \le d \le q$ that attain the bound in Theorem \ref{thm2}. Note that this classification can be obtained as a consequence of Theorem \ref{leducq}. But we present an independent and elementary proof. First, we have the following Proposition that rules out all curves that are not a union of lines. 

\begin{proposition}\label{prop1}
    If $\C$ is a projective plane curve of degree $d$, with $d \le q$, containing an irreducible component of degree at least two, then $|\C(\Fq)| \le (d-1)q + 2$. In particular, if $\C$ is not a union of lines defined over $\Fq$, then $|\C(\Fq)| < dq - d + 3$. 
\end{proposition}

\begin{proof}
   If $\C$ does not contain any lines defined over $\Fq$, then Theorem \ref{hk} applies and proves the assertion. Thus, we may assume that $\C$ contains a line defined over $\Fq$. We may write $\C = \mathcal{L} \cup \mathcal{N}$, where $\mathcal{L}$ is the union of all lines defined over $\Fq$, contained in $\C$, whereas $\mathcal{N}$ is the union of all components of $\C$ containing no lines defined over $\Fq$. Let $\deg \mathcal{N} = s$ and $\deg \mathcal{L} = d-s$. By our hypothesis, we have $s \ge 2$. Using Theorem \ref{S} and Theorem \ref{hk}, we obtain,
   $$|\mathcal{L} (\Fq)| \le (d-s)q + 1 \ \ \ \text{and} \ \ \ |\mathcal{N} (\Fq)| \le (s-1)q + 1.$$
   Consequently, $|\C(\Fq)| \le |\mathcal{L} (\Fq)| + |\mathcal{N} (\Fq)| \le (d-1)q + 2$. 

\noindent Since $d \le q$, we have $\left(dq - d + 3\right) - \left((d-1)q + 2\right) = (q - d + 1) > 0$, completing the proof.
\end{proof}

\begin{theorem}\label{thm3}
    Let $3 \le d \le q$. 
    If $\C$ is a curve defined over $\Fq$ admitting $dq - d + 3$ points in $\PP^2 (\Fq)$, then $\C$ is a union of $d$ lines $\ell_1, \dots, \ell_d$ such that $\ell_1, \dots, \ell_{d-1}$ intersect at a common point $P$ and $\ell_d$ intersect each of $\ell_1, \dots, \ell_d$ at a point different from $P$.
    In particular, there are $(q^2 + q + 1)q^2{q+1 \choose d-1}$ many distinct curves attaining the bound in Theorem \ref{thm2}. 
\end{theorem}

\begin{proof}
First, if $\C$ is not a union of lines, then Proposition \ref{prop1} shows that $|\C (\Fq)| < dq - d+ 3$. Thus, $\C$ is a union of lines. Moreover, if $\C$ is a union of at most $d-1$ lines, then Theorem \ref{S} shows that $|\C(\Fq)| \le (d-1)q+ 1 < dq - d + 3$. This shows that $\C$ is a union of $d$ distinct lines, say $\ell_1, \dots, \ell_d$. It is evident that $\C' = \C \setminus \ell_1$ is an affine curve of degree $d-1$, with $|\C'(\Fq)| = (d-1)q - (d-2)$. From Theorem \ref{CR}, we see that there are two possibilities:
\begin{enumerate}
    \item[(a)] $\ell_2, \dots, \ell_d$ meet at a common point $P$. By construction $P \notin \ell_1$. Thus, $\C$ satisfies the assertion. 
    \item[(b)] $\ell_3, \dots, \ell_d$ are parallel lines in $\AA^2$ and $\ell_2$ intersects each of $\ell_3, \dots, \ell_d$ at a point\footnote{If $d =3$, then this configuration is same as that in (a).}.  It is now evident that $\ell_1, \ell_3, \dots, \ell_d$ intersect at a common point in $\PP^2$, and $\ell_2$ intersects each of them, proving the assertion. 
\end{enumerate}
It is now trivial to count the number of curves attaining the bound. 
\end{proof}

\subsection{Towards the third highest number of points on projective plane curves} 
In this subsection, we investigate the question of determining the third highest number of $\Fq$-rational points that a projective plane curve of degree $d$ may have in $\PP^2$. In particular, if $\mathscr{C}' = \{\C \in \mathscr{C} : |\C(\Fq)| < dq - d + 3\}$, then we look for the value of  $$ \max_{\C \in \mathscr{C}'} |\C(\Fq)|.$$
A lot of work to this end has been done by A. Sboui and his collaborators (cf. \cite{RS, Sbspec, Sb2nd}). In particular, they have obtained the first three highest numbers of points on hypersurfaces of degree $d$ in $\PP^m$, which are the union of $d$ \textit{distinct} hyperplanes. Following the notations in \cite{Sbspec}, let us denote by $N_i^{\ell}$ the $i$-th highest number of $\Fq$-rational points that a curve of degree $d$ given by the union of $d$ distinct lines may have in $\PP^2$, for $i = 1, 2, 3$. In particular, for $5 \le d \le q$, we have, thanks to \cite[Theorem 3.10]{Sbspec}, 
$$N_1^{\ell} = dq + 1, \ \ N_2^{\ell} = dq - d + 3 \ \ \text{and} \ \ \ N_3^{\ell} = dq + 1 - 2 (d-3).$$
To put the above information in perspective, for $i = 1, 2, 3$, let us denote by $N_i$ the $i$-th highest number of $\Fq$-rational points on a projective plane curve of degree $d$. It is evident from Theorem \ref{S} and Theorem \ref{thm2} that for $5 \le d \le q$, we have,
$$N_1 = N_1^{\ell}  \ \ \ \text{and} \ \ \ N_2 = N_2^{\ell}.$$
Clearly, $N_3 \ge N_3^{\ell}$. Moreover, it was proved in \cite[Theorem 4.2]{Sbspec} that $N_3 = N_3^{\ell}$ for $5 \le d \le \frac{q}{3} + 2$. Thereafter, Rodier and Sboui claimed in \cite[Corollary 11]{RS} that 
when $q$ is a prime, then $N_3 = N_3^{\ell}$ for $4 < d < q-2$. We remark that this observation is incorrect. Indeed, the plane curve $\C$ of degree $d$ which is a union of $d-1$ lines $\ell_1, \dots, \ell_{d-1}$ passing through a common point $P$, with $\ell_1$ having multiplicity $2$ satisfies,
$$|\C(\Fq)| = (d-1)q + 1.$$
Note that $(d-1)q + 1 < N_2$, as observed in the proof of Theorem \ref{thm2}. Also, when $d > \frac{q+6}{2}$, we see that $N_3^{\ell} < (d-1)q+1$. Thus, the above-mentioned result fails when $q$ is a prime and $q \ge 11$. Even though the determination of $N_3$ for $3 \le d \le q$, for an arbitrary prime power $q$, is beyond reach, we present some upper bounds for $N_3$.
We are now ready to present an upper bound for $N_3$. 
\begin{theorem}\label{third}
    For $5 \le d \le q$, we have
    \begin{equation}
        N_3 = \begin{cases}
            dq + 1 - 2(d-3) \ \ &\text{if} \ \  d \le \frac{q+5}{2} \\
            \le (d-1)q + 2 \ \ \ &\text{if} \ \  d \ge \frac{q+6}{2}
        \end{cases}
    \end{equation}
\end{theorem}

\begin{proof}
    Let $\C \in \mathscr{C}'$. If $\C$ is a union of $d$ distinct lines, then $|\C(\Fq)| \le N_3^{\ell} = dq + 1 - 2(d-3)$, thanks to \cite[Theorem 3.10]{Sbspec}. It follows from Serre's inequality that if $\C$ is a union of at most $d-1$ distinct lines, then $|\C(\Fq)| \le (d-1)q + 1$. Thus, we may conclude that if $\C$ is a union of lines, then 
    \begin{equation}\label{eqmax1}
    |\C (\Fq)| \le \max \{(d-1)q + 1, dq+1 - 2(d-3)\}
    \end{equation}
    If $\C$ is not a union of lines, then Proposition \ref{prop1} implies that 
    \begin{equation}\label{notu}
        |\C(\Fq)| \le (d-1)q+2.
    \end{equation}
    Since $dq+1 - 2(d-3) \ge (d-1)q + 2$ when $d \le \frac{q+5}{2}$, the assertion follows. 
\end{proof}

\begin{remark}\label{cubicquartic}\normalfont
    We note that the Theorem \ref{third} does not apply when $d = 3$ and $d = 4$. Let us independently work out these two cases. 
\begin{enumerate}
   \item[(a)] ${\bf d = 3}:$ Let $\C$ be a projective plane cubic curve defined over $\Fq$. Suppose in addition that $\C \in \mathscr{C}'$, or equivalently, $|\C| \le 3q - 1$. If $\C$ is a union of three distinct lines defined over $\Fq$, then it is evident that $|\C(\Fq)| = 3q + 1$ or $|\C(\Fq)| = 3q$. On the other hand, if $\C$ is a union of three lines with one of them occurring at least twice, then $|\C(\Fq)| \le 2q + 1$. On the other hand, if $\C$ is irreducible, then Theorem \ref{hk} applies, and we see that $|\C(\Fq)| \le 2q + 1$. The only case to be considered here is where $\C$ is given by a union of a line $\ell$ and an irreducible quadric $Q$.  But in this case, $|\C (\Fq)| \le |\ell (\Fq)| + |Q (\Fq)| \le 2q + 2$. Furthermore, this upper bound is attained by a union of a line and a nonsingular quadric that do not intersect at any $\Fq$-rational point. In particular, we have proved that $N_3 = 2q + 2$. 
   \item[(b)] ${\bf d = 4}:$ Let $\C$ be a projective plane quartic curve defined over $\Fq$. Note that $\C \in \mathscr{C}'$ if and only if $|\C(\Fq)| < 4q - 1$. We claim that $N_3 = 4q - 2$. It is clear that $N_3 \le 4q - 2$. Consider the curve $\C$ given by a union of $4$ lines $\ell_1, \ell_2, \ell_3, \ell_4$, all defined over $\Fq$ such that for each pair $(i, j)$ with $i \neq j$, the lines $\ell_i \cap \ell_j$ at a point $P_{ij}$ and whenever $(i, j) \neq (u, v)$, we have $P_{ij} \neq P_{uv}$. It is clear that $|\C(\Fq)| = 4q - 2$, establishing our claim.   
\end{enumerate}
\end{remark}

\begin{remark}\label{fail}\normalfont

As already mentioned above, for every $d$, with $2 \le d \le q+2$, there exist curves (given by a union of $d-1$ distinct lines passing through a common point, one among them having multiplicity $2$) with exactly $(d-1)q + 1$ many points. In particular, Theorem \ref{third} implies that either $N_3 = (d-1)q+1$ or $N_3 = (d-1)q+2$, whenever $d \ge \frac{q+6}{2}$. The proof of Proposition \ref{prop1} shows that a curve $\C$, that is not a union of lines, has $(d-1)q+2$ many points if and only if $\C = \LL \cup \mathcal{N}$, where $\LL$ is the union of lines contained in $\C$, and $\mathcal{N}$ is the union of all line-free components of $\C$, satisfying:
\begin{enumerate}
    \item[(C1)] $2 \le \deg \mathcal{N}  = d-s \le d-1$,
    \item[(C2)] $|\mathcal{N} (\Fq)| = (d-s- 1)q + 1$, 
    \item[(C3)] $\deg \LL = s \ge 1$, and $\LL$ is a union of $s$ lines passing through a common point and
    \item[(C4)] $\LL (\Fq) \cap \mathcal{N}(\Fq) = \emptyset$. 
\end{enumerate}    
\end{remark}
Unfortunately, a classification of plane curves of degree $d$ with $\sqrt{q}+1 < d <q-1$ is unknown. However, partial results can be obtained to determine $N_3$. We begin with the following Lemma, which should be well-known, but we did not find a proper reference in the literature. We include a proof for the sake of completeness.

\begin{lemma}\label{sconic}
    Let $\mathcal{Q}$ be a nonsingular conic, and $\LL$ be a union of $s$ distinct lines passing through a common $\Fq$-rational point $P$. If $(\LL \cap \mathcal{Q})(\Fq)$ is empty, then $s \le \frac{q+1}{2}$. 
\end{lemma}

\begin{proof}
   By hypothesis $P \notin \mathcal{Q}$. Note that the $q + 1$ lines through $P$ cover all the points on $\mathcal{Q}$. Since $\mathcal{Q}$ contains no lines, there must be at least $\frac{q+1}{2}$ lines through $P$ that intersect $\mathcal{Q}$ at an $\Fq$-rational point. Thus $s \le \frac{q+1}{2}$. 
\end{proof}
\begin{proposition}
    Suppose $5 \le d \le q-1$. If $q = p$ or $q = p^2$, then 
   $$N_3 = 
        \begin{cases}
            dq+1 - 2(d-3) \ \ \mathrm{if} \ \ d \le \frac{q+5}{2} \\
            (d-1)q + 1 \ \ \ \ \mathrm{if} \ \ d \ge \frac{q+6}{2}
        \end{cases}.$$
\end{proposition}

\begin{proof}
The value of $N_3$ when $5 \le d \le \frac{q+5}{2}$ is already determined in Theorem \ref{third}. Thus we investigate the case when $d \ge \frac{q+6}{2}$. Let $\C$ be a projective plane curve of degree $d$ and $\C = \LL \cup \N$, where $\LL$ denotes the union of lines contained in $\C$, whereas $\N$ is the line free component of $\C$. Suppose $\deg \LL = s$. It is evident from Theorem \ref{hk} that if $s = 0$, then $|\C (\Fq)| \le (d-1)q + 1$, and we are done. So we may assume that $s \ge 1$. In view of Remark \ref{fail}, we suppose that $\LL$ and $\N$ satisfy (C1) - (C3). We will show that in both cases, when $q = p$ or $q = p^2$, the condition (C4) can not be satisfied. 
First, assume that $d - s = 2$. Since $d \ge \frac{q+6}{2}$, it follows that $s \ge \frac{q+2}{2}$. Lemma \ref{sconic} shows that $\LL(\Fq) \cap \N(\Fq) \neq \emptyset$.
   \begin{enumerate}
       \item[(a)] If $q = p$ and $d \le q-1$, from Remark \ref{attained}, we see that $\N$ must be nonsingular plane conic. As observed above, in this case (C1)-(C4) can not be satisfied simultaneously, and hence $|\C(\Fq)| \le (d-1)q + 1$.  
       \item[(b)] If $q = p^2$, we have $d - s = 2$ or $d - s = p+1$. When $d -s = 2$,  we are done. If $d - s = p+1$, the curve $\mathcal{N}$ is, in fact, a Hermitian curve. Since any line in $\PP^2(\Fq)$ intersects the Hermitian curve at an $\Fq$-rational point, we see that (C4) is violated. Thus $|\C(\Fq)| \le (d-1)q + 1$.  
   \end{enumerate}
   This completes the proof.
\end{proof}
We leave it to the reader to translate the results of Section 3 and Section 4 in terms of the determination of the corresponding weights of the projective Reed-Muller codes. As it is clear from the context, the complete determination of the second and third weights of the projective Reed-Muller codes is intrinsically related to the determination of good upper bounds on the number of $\Fq$-rational points on hypersurfaces that are absolutely irreducible. We remark that these problems are difficult problems and will require significant research in the future. 

\section{Acknowledgement}
The author is grateful to Professor Bhaskar Bagchi for some intriguing discussions on the topic of this article. The author thanks the anonymous referee for their careful reading of the manuscript and for providing some invaluable suggestions for improving the paper.

\end{document}